\documentclass[letterpaper, 10pt, conference]{Classes/ieeeconf}
\IEEEoverridecommandlockouts
\usepackage{Packages/decar-common}
\usepackage{Packages/decar-dynamics}
\usepackage{Packages/decar-lie}
\usepackage{Packages/decar-post}
\usepackage{Packages/myFormat}

\title{\LARGE \bf Input-Output Stability of Gradient Descent: A Discrete-Time Passivity-Based Approach}

\author{Sepehr Moalemi$^{1}$ and James Richard Forbes$^{1}$
    \thanks{$^{1}$Sepehr Moalemi {\tt\small sepehr.moalemi@mail.mcgill.ca} and James Richard Forbes {\tt\small james.richard.forbes@mcgill.ca} are with the Department of Mechanical Engineering, McGill University, 817 Sherbrooke St. W., Montreal, QC H3A 0C3, Canada.}
    \thanks{%
        Funding provided by the NSERC Discover Grant program.
    }%
}

\begin{document}
    \maketitle
    \thispagestyle{empty}
    \pagestyle{empty}

    \fontdimen16\textfont2=\fontdimen17\textfont2
    \fontdimen13\textfont2=5pt
    \begin{abstract}
    This paper presents a discrete-time passivity-based analysis of the gradient descent method for a class of functions with sector-bounded gradients. Using a loop transformation, it is shown that the gradient descent method can be interpreted as a passive controller in negative feedback with a very strictly passive system. The passivity theorem is then used to guarantee input-output stability, as well as the global convergence, of the gradient descent method. Furthermore, provided that the lower and upper sector bounds are not equal, the input-output stability of the gradient descent method is guaranteed using the weak passivity theorem for a larger choice of step size. Finally, to demonstrate the utility of this passivity-based analysis, a new variation of the gradient descent method with variable step size is proposed by gain-scheduling the input and output of the gradient.
\end{abstract}
    \section{Introduction}
    Optimization is an indispensable tool used in many branches of engineering. Recently, popular first-order optimization methods such as gradient decent, Polyak's heavy-ball \cite{polyak_hb}, and Nesterov's accelerated method \cite{Nesterov}, have been analyzed through the lens of feedback control in a discrete-time setting.
    
    In particular, in~\cite{ugrinovskii}, for a function with sector-bounded gradient, the circle criterion is used to derive expressions for the parameters of the heavy ball method that guarantee its global convergence. This result was extended in~\cite{alex_petersen} by incorporating an additional parameter, leading to the new generalized accelerated gradient method. The small gain theorem is used in~\cite{hu_lessard} to relate the input-output gain computations of a defined complementary sensitivity function to the convergence properties of gradient descent (GD) methods. Additionally,~\cite{hu_lessard} demonstrates a connection between the step size selection of the gradient method and \(\mathcal{H}_{\infty} \) state feedback synthesis. For a class of convex constraint-coupled optimization problems, a passivity-based analysis of the alternating direction method of multipliers (ADMM) is provided in~\cite{Notarnicola_passivity_admm}. Dissipativity theory and Lyapunov-based approaches for algorithm analysis have also been studied in~\cite{lessard_dissipativity, bryan_lessard_Lyapunov}. Other robust control tools such as integral quadratic constraints~(IQCs) have been used in the analysis and synthesis of novel first-order iterative optimization methods. Within the context of optimization, the IQC framework is first used in~\cite{lessard_recht_iqc} to derive numerical upper bounds on convergence rates of first-order methods. For the class of \(m\)-strongly convex and \(L\)-smooth functions, a graphical design procedure based on IQCs is used to derive the optimal parameters for the triple momentum method introduced in~\cite{triple_momentum}. The IQC framework is also used in~\cite{seiler_iqc} to perform an analysis of GD with varying step sizes, where the algorithm is modeled as a linear parameter-varying (LPV) system.

    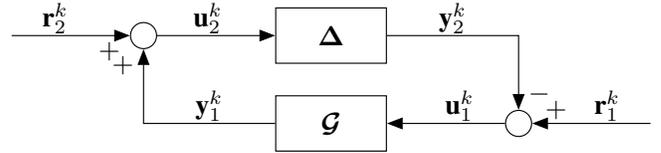
\begin{figure}[t]
        \centering
        \vspace{2pt}
    \resizebox{1.0\columnwidth}{!}{\begin{tikzpicture}[auto, node distance=2cm]

    \def\blockWidth{4em}                 
    \def\blockHeight{2em}                
    \def\blockVerticalSpacing{0.4cm}    

    \def\arrowTipLength{2mm}             
    \def\arrowTipWidth{1.5mm}            

    \def\smallArrowLength{1.5cm}         
    \def\smallArrowStretch{1.5}          
    \def\smallArrowShrink{0.8}           

    \tikzset{     
        block/.style = {draw, fill=none, rectangle, 
                        minimum height=\blockHeight, minimum width=\blockWidth}
    }

    \tikzset{    
        sum/.style  = {draw, fill=none, circle, node distance=1cm},
    }

    \tikzset{    
        tmp/.style    = {coordinate}, 
        input/.style  = {coordinate},
        output/.style = {coordinate}
    }

    \tikzset{      
        pinstyle/.style = {pin edge={to-,thin}}
    }

    \tikzset{     
        myarrow/.style = {->, -{Triangle[length=\arrowTipLength, width=\arrowTipWidth]}}
    }

    \node [block] (Nonlinearity) {\(\mbs{\Delta}\)};
    \node [block, , below=\blockVerticalSpacing of Nonlinearity.south, anchor=north](G) {\(\bm{\mathcal{G}}\)};

    \node [tmp, right=\smallArrowLength of G.east] (G_in) {};
    \node [tmp, left=\smallArrowLength  of G.west] (G_out) {};

    \node [sum, left=\smallArrowLength of Nonlinearity.west] (sum_node_up){};
    \node [sum, right=\smallArrowLength of G.east]           (sum_node_down){};


    \draw [myarrow](G) -| node[pos=0.25, above, yshift=-0.09cm]{$\mbf{y}^{k}_1$} (sum_node_up)node[at end, left, yshift=-\arrowTipLength] {$+$};
    \draw [myarrow](sum_node_down) -- (G)node[at end, above, xshift=\smallArrowLength/2 + 5, yshift=-0.09cm] {$\mbf{u}^{k}_1$};
    \draw [myarrow](sum_node_up) -- (Nonlinearity)node[midway, above, xshift=-3, yshift=-0.09cm] {$\mbf{u}^{k}_2$};
    \draw [myarrow](Nonlinearity.east) -| node[pos=0.25, above, yshift=-0.09cm]{$\mbf{y}^{k}_2$}(sum_node_down)node[at end, right, yshift=\arrowTipLength] {$-$};

    \draw [myarrow] ([xshift=\smallArrowLength+5] sum_node_down.center) -- node[above, midway, xshift=\arrowTipLength, yshift=-0.09cm] {$\mbf{r}^{k}_1$}(sum_node_down) node[at end, above, xshift=\arrowTipLength+3, yshift=-2] {$+$};
    \draw [myarrow] ([xshift=-\smallArrowLength-5] sum_node_up.center) -- node[above, midway, xshift=-\arrowTipLength, yshift=-0.09cm] {$\mbf{r}^{k}_2$}(sum_node_up) node[at end, left, yshift=-\arrowTipLength] {$+$};
\end{tikzpicture}}%

        \vspace{-15pt}
        \caption{The negative feedback interconnection of two systems \(\bm{\mathcal{G}}\) and \(\mbs{\Delta}\).}
        \label{fig:negative_feedback}
        \vspace{-10pt}
    \end{figure}
    This paper considers a general class of possibly non-convex functions with a unique global minimizer and a sector-bounded gradient, with lower bound $m$ and upper bound $L$. Inspired by~\cite{hu_lessard, lessard_recht_iqc}, the GD method is recast as a discrete-time linear time-invariant (LTI) controller \(\bm{\mathcal{G}}\) in negative feedback with a static memoryless nonlinearity \(\mbs{\Delta}\), as per \Cref{fig:negative_feedback}. It is shown that the GD controller \(\bm{\mathcal{G}}\), as presented in~\cite{lessard_recht_iqc}, is strictly proper and, therefore, \emph{not} passive. As such, a loop transformation is performed, adding a positive feedthrough term to \(\bm{\mathcal{G}}\) and a positive feedback term to \(\mbs{\Delta}\). It is then shown that the GD method can be interpreted as a passive controller in negative feedback with a very strictly passive (VSP) system, provided the step size, \(\alpha\), satisfies \(\alpha \in \mleft(0, 2/L\mright)\). The \emph{strong passivity theorem} is then used to guarantee the input-output stability, as well as global convergence, of the GD method. Provided the lower and upper sector bounds are not equal, input-output stability of the GD method can still be guaranteed for \(\alpha = 2/L\), using the \emph{weak passivity theorem}.
    Finally, using a discrete-time adaptation of the gain-scheduling architecture in~\cite{Damaren_passive_map}, a new variation of the GD method with varying step size is proposed, where the input and output of the gradient are scaled using the scheduling functions. In summary, the novel contributions of this work are as follows:
    \begin{enumerate}
        \item {%
            Presenting a discrete-time passivity-based analysis of the GD method for a general class of possibly non-convex functions, having a unique global minimizer and a sector-bounded gradient.
        }%
        \item{%
            Demonstrating input-output stability of the GD method for a larger step size using the weak passivity theorem.
        }%
        \item{%
            Introducing a variation of the GD method with variable step size, using a passivity-based gain-scheduled control architecture.
        }%
    \end{enumerate}

    The remainder of this paper is as follows. Preliminaries and control interpretation of the GD method are presented in \Cref{sec:preliminaries}. The loop transformation and passivity-based analysis of the GD method are presented in \Cref{sec:main_results}. A discussion on the connection between the input-output stability of the GD method and the convergence of the algorithm, as well as the extension of the passivity-based analysis to the case of time-varying step sizes is presented in \Cref{sec:discussion}. Numerical results are shown in \Cref{sec:simulation}, followed by closing remarks in \Cref{sec:closing_remarks}.
    \section{Preliminaries} \label{sec:preliminaries}
\subsection{Notation}
    Scalars, column matrices, and matrices are denoted as \mbox{\(\alpha \in \mathbb{R}\)}, \(\mbf{v} \in \mathbb{R}^{n}\), and \(\mbf{A} \in \mathbb{R}^{n \times m}\), respectively. Operators are denoted as \(\bm{\mathcal{G}}\) and function classes are denoted as \(\mathcal{F}\). A positive definite matrix is denoted as \(\mbf{A} \succ 0\) and a negative semidefinite matrix is denoted as \(\mbf{A} \preceq 0\). The identity and zero matrices are \(\eye\) and \(\mbf{0} \), respectively.
\subsection{Spaces and Operators}
    \begin{definition}[Truncation operator {~\cite{feedback_systems}}]
        For a function \(\mbf{u} : \mathbb{Z}_{\geq 0} \to \mathbb{R}^{n}\) and \(T \in \mathbb{Z}_{\geq 0}\), its truncation, \(\mbf{u}_T\), is defined as \(\mbf{u}_T(k) = \mbf{u}(k)\) for \(k < T\) and \(\mbf{u}_T(k) = \mbf{0}\) for \(t \geq T\).
    \end{definition}
    \vspace{2pt}
    \begin{definition}[Truncated inner product {~\cite{feedback_systems}}]
        For functions \(\mbf{u}, \mbf{y} : \mathbb{Z}_{\geq 0} \to \mathbb{R}^{n}\), the truncated inner product over the discrete-time interval \(\mathcal{T} = \mleft\{0, 1, \hdots, T-1 \mright\}\) is defined as \(\langle \mbf{u}, \mbf{y}\rangle_T = \langle\mbf{u}_T, \mbf{y}_T\rangle = \sum_{k \in \mathcal{T}} \mbf{u}^{\trans}(k) \mbf{y}(k), \forall T \in \mathbb{Z}_{> 0}\).
    \end{definition}
    \vspace{2pt}
    \begin{definition}[$\ell_{2e}$ and $\ell_{2}$ function spaces {~\cite{feedback_systems}}] \label{def:Lp_spaces}
        Given a function \(\mbf{u} : \mathbb{Z}_{\geq 0} \to \mathbb{R}^{n} \), \(\mbf{u} \in \ell_{2e}\) if \(\mleft\| \mbf{u} \mright\|_{2T}^2 = \langle \mbf{u}, \mbf{u}\rangle_T < \infty\), for all \(T \in \mathbb{Z}_{>0}\). Moreover, \(\mbf{u} \in \ell_{2}\) if \(\mleft\| \mbf{u} \mright\|_{2}^2 = \langle \mbf{u}, \mbf{u}\rangle < \infty\).
    \end{definition}
    \vspace{2pt}
    \begin{definition}[$\ell_{2}$ stability {~\cite{feedback_systems}}] \label{def:Lp_stability}
        The operator \(\bm{\mathcal{G}} : \ell_{2e} \to \ell_{2e}\) is \(\ell_{2}\)-stable if \(\bm{\mathcal{G}}\mbf{u} \in \ell_{2}\) for all \(\mbf{u} \in \ell_{2}\).
    \end{definition}
    %
\subsection{Passivity Theorem and Definitions}
    \begin{definition}[Passivity {~\cite{feedback_systems}}] \label{def:passivity}
        Consider a square system with input \(\mbf{u} \in \ell_{2e}\) and output \(\mbf{y} \in \ell_{2e}\) mapped through the operator \(\bm{\mathcal{G}} : \ell_{2e} \to \ell_{2e}\). The system \(\bm{\mathcal{G}}\) is 
        \begin{enumerate}[label=\textit{\roman*})]
            \item {%
            \emph{passive} if \(\exists \beta \in \mathbb{R}_{\leq 0} \) such that
            \begin{equation*}
                \langle \mbf{u}, \mbf{y}\rangle_T 
                \geq
                \beta
                , \quad \forall \mbf{u} \in \ell_{2e},\,\forall T \in \mathbb{Z}_{>0},
            \end{equation*}
            }%
            \item {%
            \emph{input strictly passive (ISP)} if \(\exists \beta \in \mathbb{R}_{\leq 0} \) and \(\exists \delta \in \mathbb{R}_{>0} \) such that
            \begin{equation*}
                \langle \mbf{u}, \mbf{y}\rangle_T 
                \geq
                \beta + \delta \mleft\| \mbf{u} \mright\|^{2}_{2T}
                , \quad \forall \mbf{u} \in \ell_{2e},\,\forall T \in \mathbb{Z}_{>0},
            \end{equation*}
            }%
            \item {%
            \emph{very strictly passive (VSP)} if \(\exists \beta \in \mathbb{R}_{\leq 0} \) and \(\exists \delta, \varepsilon \in \mathbb{R}_{>0} \) such that
            \begin{equation*}
                \langle \mbf{u}, \mbf{y}\rangle_T 
                \geq
                \beta + \delta \mleft\| \mbf{u} \mright\|^{2}_{2T} + \varepsilon \mleft\| \mbf{y} \mright\|^{2}_{2T}
                , \quad \forall \mbf{u} \in \ell_{2e},\,\forall T \in \mathbb{Z}_{>0}.
            \end{equation*}
            }%
        \end{enumerate}
    \end{definition}
    \vspace{4pt}
    \begin{lemma}[Positive real {~\cite[Lemma 3]{anderson}}] \label{lemma:positive_real}
        Let \(\mbf{G}(z)\) be a square matrix of real rational functions of \(z\) and let \(\mleft(\mbf{A}, \mbf{B}, \mbf{C}, \mbf{D}\mright) \) be a minimal realization of \(\mbf{G}(z)\). Then \(\mbf{G}(z)\) is \emph{positive real} if and only if \(\exists\mbf{P} = \mbf{P}^{\trans} \succ 0\) such that
        \begin{equation}\label{eqn:positive_real}
            \begin{bmatrix}
                \mbf{A}^{\trans} \mbf{P} \mbf{A} - \mbf{P} & \mbf{A}^{\trans} \mbf{P} \mbf{B} - \mbf{C}^{\trans} \\
                \mleft( \mbf{A}^{\trans} \mbf{P} \mbf{B} - \mbf{C}^{\trans} \mright)^{\trans}  & \mbf{B}^{\trans} \mbf{P} \mbf{B} - \mleft( \mbf{D} + \mbf{D}^{\trans} \mright)
            \end{bmatrix}
            \preceq 0.
        \end{equation}
    \end{lemma}
    \begin{remark}\label{remark:positive_real_is_passive}
        In \Cref{lemma:positive_real}, a necessary and sufficient condition for a positive real transfer matrix is given. However, an LTI system is passive if and only if it has a positive real transfer function~\cite[Theorem 13.27]{haddad_Chellaboina}.
    \end{remark}
    \begin{theorem}[Weak passivity theorem~\cite{feedback_systems}]\label{thrm:weak_passivity}
        Consider two systems \(\bm{\mathcal{G}} : \ell_{2e} \to \ell_{2e}\) and \(\mbs{\Delta} : \ell_{2e} \to \ell_{2e}\) in negative feedback as per \Cref{fig:negative_feedback} with \(\mbf{r}_2 = \mbf{0}\). If \(\bm{\mathcal{G}}\) is passive, \(\mbs{\Delta}\) is ISP, and \(\mbf{r}_1 \in \ell_{2}\), then \(\mbf{y}_1 \in \ell_{2}\).
    \end{theorem}
    \begin{theorem}[Strong passivity theorem~\cite{feedback_systems}]\label{thrm:strong_passivity}
        Consider two systems \(\bm{\mathcal{G}} : \ell_{2e} \to \ell_{2e}\) and \(\mbs{\Delta} : \ell_{2e} \to \ell_{2e}\) in negative feedback as per \Cref{fig:negative_feedback}. If \(\bm{\mathcal{G}}\) is passive, \(\mbs{\Delta}\) is VSP, and \(\mbf{r}_1, \mbf{r}_2 \in \ell_{2}\), then \(\mbf{y}_1, \mbf{y}_2 \in \ell_{2}\).
    \end{theorem}
\subsection{Function Classes}
    Denote the set of continuously differentiable functions \mbox{$f: \mathbb{R}^n \rightarrow \mathbb{R}$} that are $m$-strongly convex and $L$-smooth with \(0 < m \leq L\) as \(\mathcal{F}_{m, L}\).
    \vspace{1pt}
    \begin{remark}
        Consider a function \(f \in \mathcal{F}_{m, L}\) and its unique global minimizer \(\mbf{x}^{*}\) such that \(\nabla f(\mbf{x}^\ast) = \mbf{0}\). The function \mbox{\(g(\mbf{x}) = f(\mbf{x}) - \frac{m}{2} \mleft\| \mbf{x} \mright\|_2^2\)} is convex and \(\mleft( L - m \mright)\)--smooth. The co-coercivity of \(\nabla g\) can be written as
        \begin{equation} \label{eqn:co_coercivity}
            \begin{split}
                \hspace{-5pt}
                \left\langle\mbf{x} - \mbf{x}^\ast, \nabla f(\mbf{x})\right\rangle 
                \geq 
                \frac{mL}{m + L}  \mleft\| \mbf{x} - \mbf{x}^\ast \mright\|_2^2 
                +
                \frac{1}{m + L}  \mleft\| \nabla f(\mbf{x}) \mright\|_2^2,
            \end{split}
        \end{equation}
        for all \(\mbf{x} \in \mathbb{R}^n\)~\cite[Proposition 5]{lessard_recht_iqc}.
    \end{remark}

    The inequality in \cref{eqn:co_coercivity} encloses \(\nabla f(\mbf{x})\) in the sector between two lines with slopes \(m\) and \(L\)~\cite{lessard_dissipativity}. More generally, the set of continuously differentiable functions with unique global minimizers satisfying \cref{eqn:co_coercivity} is denoted as \(\mathcal{S}_{m, L}\). This class of functions has sector-bounded gradients and contains non-convex functions as-well, meaning \(\mathcal{F}_{m, L} \subset \mathcal{S}_{m, L}\)~\cite{hu_lessard}. 
\subsection{Optimization Problem}
    Consider the unconstrained optimization problem
    \begin{equation} \label{eq:optimization_problem}
        \min_{\mbf{x} \in \mathbb{R}^{n}} f(\mbf{x}),
    \end{equation}
    where \( f \in \mathcal{S}_{m, L}\). The GD method with the update rule
    \begin{equation} \label{eq:GD}
        \mbf{x}(k+1) = \mbf{x}(k) - \alpha \nabla f(\mbf{x}(k))
    \end{equation}
    can be used to solve \eqref{eq:optimization_problem} using an initialization \(\mbf{x}(0)\in \mathbb{R}^{n}\) and a constant step size \(\alpha \in \mathbb{R}_{>0}\).
\subsection{Control Interpretation of Gradient Descent}\label{subsec:control_interpretation}
    The discrete-time feedback representation of GD is discussed in~\mbox{\cite[Section 2]{lessard_recht_iqc}}. By modifying the approach in~\cite{lessard_recht_iqc}, the GD method can be interpreted as a static memoryless nonlinearity \(\mbs{\Delta} = \nabla f\) in negative feedback with an LTI dynamical system \(\bm{\mathcal{G}}\), as shown in \Cref{fig:negative_feedback}. Henceforth, the notation \(\mbf{x}(k)\) is abbreviated to \(\mbf{x}^{k}\).
    \begin{remark}
        Similar to~\cite{hu_lessard}, it is assumed that \(\mbs{\Delta}\) in \Cref{fig:negative_feedback} maps \(\mbf{u}_2\) to \(\mbf{y}_2\) as \(\mbf{y}_{2}^{k} = \nabla f(\mbf{u}_{2}^{k} + \mbf{x}^\ast)\), such that for \(\mbf{u}_{2}^{k} = \mbf{0}\), \(\mbf{y}_{2}^{k} = \nabla f(\mbf{x}^\ast) = \mbf{0}\). 
        Using the shifted state \(\mbs{\xi}^{k} = \mbf{x}^{k} - \mbf{x}^\ast\), the LTI system \(\bm{\mathcal{G}}\) in \Cref{fig:negative_feedback} can be represented as
        \begin{subequations} \label{eq:G}
            \begin{align}
                \mbs{\xi}^{k+1} &= \mbf{A} \mbs{\xi}^{k} + \mbf{B} \mbf{u}_{1}^{k}, \label{eq:state_space_G} \\%
                \mbf{y}_{1}^{k} &= \mbf{C} \mbs{\xi}^{k} + \mbf{D} \mbf{u}_{1}^{k}. \label{eq:output_G}%
            \end{align}
        \end{subequations}
    \end{remark}

    The GD update rule in \eqref{eq:GD} can then be written in the form of \cref{eq:G} with minimal realization \(\mleft( \mbf{A}, \mbf{B}, \mbf{C}, \mbf{D} \mright) = \mleft( \eye, \alpha \eye, \eye, \mbf{0} \mright)\) as shown in \Cref{fig:gd_neg_feedback}. Additionally, the convergence rate of the GD method becomes analogous to the rate at which \mbox{\(\mbs{\xi}^{k} \to \mbf{0}\)}.
    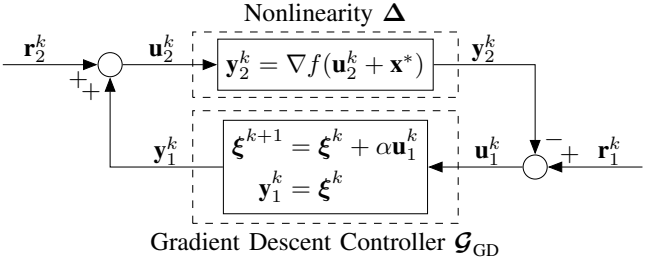
\begin{figure}[t]
        \centering
        \vspace{1pt}
    \resizebox{1.0\columnwidth}{!}{\begin{tikzpicture}[auto, node distance=2cm]

    \def\blockWidth{4em}                 
    \def\blockHeight{2em}                
    \def\blockVerticalSpacing{0.4cm}    

    \def\arrowTipLength{2mm}             
    \def\arrowTipWidth{1.5mm}            

    \def\dashedRectangleWidth{3.65cm}
    \def\dashedRectangleHeight{1.55cm}
    \def\dashedRectanglCenterXShift{0mm}
    \def\dashedRectanglCenterYShift{0cm}
    \def\dashedtextxhift{4mm}     
    
    \def\NdashedRectangleWidth{3.65cm}
    \def\NdashedRectangleHeight{0.9cm}
    \def\NdashedRectanglCenterXShift{0mm}
    \def\NdashedRectanglCenterYShift{0cm}
    \def\Ndashedtextxhift{5.1mm}                 

    \def\smallArrowLength{1.3cm}         
    \def\smallArrowStretch{1.5}          
    \def\smallArrowShrink{0.8}           

    \tikzset{     
        block/.style = {draw, fill=none, rectangle, 
                        minimum height=\blockHeight, minimum width=\blockWidth}
    }

    \tikzset{    
        sum/.style  = {draw, fill=none, circle, node distance=1cm},
    }

    \tikzset{    
        tmp/.style    = {coordinate}, 
        input/.style  = {coordinate},
        output/.style = {coordinate}
    }

    \tikzset{      
        pinstyle/.style = {pin edge={to-,thin}}
    }

    \tikzset{     
        myarrow/.style = {->, -{Triangle[length=\arrowTipLength, width=\arrowTipWidth]}}
    }

    \node [block] (Nonlinearity) {\(\mbf{y}_{2}^{k} = \nabla f(\mbf{u}_{2}^k + \mbf{x}^\ast)\)};
    \node [block, below=\blockVerticalSpacing of Nonlinearity.south, anchor=north](G) {        
        \(
            \begin{aligned}
                \mbs{\xi}^{k+1} &= \mbs{\xi}^{k} + \alpha \mbf{u}_{1}^{k}
                \\%
                \mbf{y}_{1}^{k}     &= \mbs{\xi}^{k}
            \end{aligned}
        \)
    };

    \node [tmp, right=\smallArrowLength of G.east] (G_in) {};
    \node [tmp, left=\smallArrowLength  of G.west] (G_out) {};

    \node [sum, left=\smallArrowLength of Nonlinearity.west] (sum_node_up){};
    \node [sum, right=\smallArrowLength of G.east]           (sum_node_down){};


    \draw [myarrow](G) -| node[pos=0.25, above, yshift=-0.09cm]{$\mbf{y}_1^k$} (sum_node_up)node[at end, left, yshift=-\arrowTipLength] {$+$};
    \draw [myarrow](sum_node_down) -- (G)node[at end, above, xshift=\smallArrowLength/2 + 5, yshift=-0.09cm] {$\mbf{u}_1^k$};
    \draw [myarrow](sum_node_up) -- (Nonlinearity)node[midway, above, xshift=-3, yshift=-0.09cm] {$\mbf{u}_2^k$};
    \draw [myarrow](Nonlinearity.east) -| node[pos=0.25, above, yshift=-0.09cm]{$\mbf{y}_2^k$}(sum_node_down)node[at end, right, yshift=\arrowTipLength] {$-$};

    \draw [myarrow] ([xshift=\smallArrowLength+5] sum_node_down.center) -- node[above, midway, xshift=\arrowTipLength, yshift=-0.09cm] {$\mbf{r}_1^k$}(sum_node_down) node[at end, above, xshift=\arrowTipLength+3, yshift=-2] {$+$};
    \draw [myarrow] ([xshift=-\smallArrowLength-5] sum_node_up.center) -- node[above, midway, xshift=-\arrowTipLength, yshift=-0.09cm] {$\mbf{r}_2^k$}(sum_node_up) node[at end, left, yshift=-\arrowTipLength ] {$+$};

    \draw[dashed] (G.center) ++(-\dashedRectangleWidth/2 - \dashedRectanglCenterXShift, -\dashedRectangleHeight/2 - \dashedRectanglCenterYShift) rectangle ++(\dashedRectangleWidth, \dashedRectangleHeight);
    \node[below, yshift= -\dashedRectangleHeight/2] at (G.center) {\textrm{Gradient Descent Controller }\(\bm{\mathcal{G}}_{\textrm{GD}}\)};

    \draw[dashed] (Nonlinearity.center) ++(-\NdashedRectangleWidth/2 - \NdashedRectanglCenterXShift, -\NdashedRectangleHeight/2 + \NdashedRectanglCenterYShift) rectangle ++(\NdashedRectangleWidth, \NdashedRectangleHeight);
    \node[below, yshift= \NdashedRectangleHeight/2 + \Ndashedtextxhift] at (Nonlinearity.center) {\textrm{Nonlinearity} \(\mbs{\Delta}\)}; 
\end{tikzpicture}}%

        \vspace{-22pt}
        \caption{The negative feedback interconnection of the GD controller, \(\bm{\mathcal{G}}_{\textrm{GD}}\), and the nonlinearity \(\mbs{\Delta}\), being the shifted gradient that maps zero inputs to zero outputs.}
        \label{fig:gd_neg_feedback}
        \vspace{-15pt}
    \end{figure}
    \section{Main Contribution} \label{sec:main_results}
\begin{figure}[t]
    \centering
    \resizebox{1.0\columnwidth}{!}{\begin{tikzpicture}[auto, node distance=2cm]

    \def\blockWidth{4em}                 
    \def\blockHeight{2em}                
    \def\blockVerticalSpacing{3.3cm}     

    \def\arrowTipLength{2mm}             
    \def\arrowTipWidth{1.5mm}            

    \def\dashedRectangleWidth{5.3cm}
    \def\dashedRectangleHeight{2.45cm}
    \def\dashedRectanglCenterXShift{2.2mm}
    \def\dashedRectanglCenterYShift{-0.45cm}
    \def\dashedtextxhift{5mm}        

    \def\NdashedRectangleWidth{5.3cm}
    \def\NdashedRectangleHeight{1.9cm}
    \def\NdashedRectanglCenterXShift{2.2mm}
    \def\NdashedRectanglCenterYShift{-0.48cm}
    \def\Ndashedtextxhift{0.5mm}                 

    \def\smallArrowLength{0.7cm}           
    \def\smallArrowStretch{1.5}          
    \def\smallArrowShrink{0.7}           

    \tikzset{     
        block/.style = {draw, fill=none, rectangle, 
                        minimum height=\blockHeight, minimum width=\blockWidth}
    }

    \tikzset{    
        sum/.style  = {draw, fill=none, circle, node distance=1cm},
    }

    \tikzset{    
        tmp/.style    = {coordinate}, 
        input/.style  = {coordinate},
        output/.style = {coordinate}
    }

    \tikzset{      
        pinstyle/.style = {pin edge={to-,thin}}
    }

    \tikzset{     
        myarrow/.style = {->, -{Triangle[length=\arrowTipLength, width=\arrowTipWidth]}}
    }

    \node [block] (Nonlinearity) {\(\mbf{y}_{2}^{k} = \nabla f(\mbf{u}_{2}^k + \mbf{x}^\ast)\)};
    \node [block, below=\blockVerticalSpacing of Nonlinearity.center, anchor=center](G) { 
        \(
            \begin{aligned}
                \mbs{\xi}^{k+1} &= \mbs{\xi}^{k} + \alpha \mbf{u}_{1}^{k}
                \\%
                \mbf{y}_{1}^{k}     &= \mbs{\xi}^{k}
            \end{aligned}
        \)
    };

    \node [block, below=\blockVerticalSpacing/6 of Nonlinearity.south, anchor=center] (D_sub) {\(D \eye\)};
    \node [block, above=\blockVerticalSpacing/6 of G.north, anchor=center] (D_add) {\(D \eye\)};

    \node [tmp, right=\smallArrowLength  of G.east] (D_add_in) {};
    \node [tmp, right=\smallArrowLength  of G.east, yshift=\blockVerticalSpacing] (D_sub_in) {};
    \node [tmp, left=4*\smallArrowLength of G.west] (G_out) {};

    \node [sum, right=\smallArrowLength*\smallArrowShrink of D_add_in] (sum_node_right){};
    \node [sum, left=\smallArrowLength of G.west] (sum_node_D_add){};
    \node [sum, left=\smallArrowLength*\smallArrowShrink of sum_node_D_add.west, yshift=\blockVerticalSpacing] (sum_node_up){};
    \node [sum, left=\smallArrowLength of G.west, yshift=\blockVerticalSpacing] (sum_node_D_sub){};

    \draw [myarrow](Nonlinearity.east) -| node[pos=0.25, yshift=-0.09cm]{$\mbf{y}^{k}_{2}$}(sum_node_right)node[at end, right, yshift=\arrowTipLength] {$-$};

    \draw[myarrow] (sum_node_right) -- (G)node[pos=0.5, xshift=5pt, yshift=+0.09cm]{$\mbf{u}^{k}_{1}$};

    \draw[myarrow] (D_add_in) |- (D_add) {};
\draw[myarrow] (D_sub_in) |- (D_sub) {};

    \draw[myarrow] (G) -- (sum_node_D_add)node[at end, right, yshift=\arrowTipLength] {$+$}node[at end, midway, xshift=2pt, yshift=0.09cm] {$\mbf{y}_1^k$};

    \draw[myarrow] (sum_node_D_sub) -- (Nonlinearity)node[midway, above, xshift=-3pt, yshift=-0.09cm] {$\mbf{u}^{k}_{2}$}; {};

    \draw [myarrow] (D_add) -| (sum_node_D_add) node[at end, right, yshift=\arrowTipLength] {$+$};
    \draw [myarrow] (D_sub) -| (sum_node_D_sub) node[at end, left, yshift=-\arrowTipLength] {$+$};

    \draw [myarrow] (sum_node_D_add) -| (sum_node_up) node[at end, left, yshift=-\arrowTipLength] {$+$};
    \draw [myarrow] (sum_node_up) -- (sum_node_D_sub) node[at end, left, yshift=-\arrowTipLength] {$+$};

    \draw [myarrow] ([xshift=\smallArrowLength*2] sum_node_right.center) -- node[below, midway, xshift=\arrowTipLength, yshift=0.09cm] {$\mbf{r}_1^k$}(sum_node_right) node[at end, above, xshift=\arrowTipLength+3, yshift=-2] {$+$};
    \draw [myarrow] ([xshift=-\smallArrowLength*2] sum_node_up.center) -- node[above, midway, xshift=-3pt, yshift=-0.09cm] {$\bar{\mbf{r}}_2^k$}(sum_node_up) node[at end, left, yshift=-\arrowTipLength] {$+$};

    \draw[dashed] (G.center) ++(-\dashedRectangleWidth/2 - \dashedRectanglCenterXShift, -\dashedRectangleHeight/2 - \dashedRectanglCenterYShift) rectangle ++(\dashedRectangleWidth, \dashedRectangleHeight);
    \draw[dashed] (Nonlinearity.center) ++(-\NdashedRectangleWidth/2 - \NdashedRectanglCenterXShift, -\NdashedRectangleHeight/2 + \NdashedRectanglCenterYShift) rectangle ++(\NdashedRectangleWidth, \NdashedRectangleHeight);

    \node[below, yshift= -\dashedRectangleHeight/2 + \dashedtextxhift] at (G.center) {\textrm{Modified Gradient Descent Controller }\(\bar{\bm{\mathcal{G}}}_{\textrm{GD}}\)};
    \node[below, yshift= \NdashedRectangleHeight/2 + \Ndashedtextxhift] at (Nonlinearity.center) {\textrm{Nonlinearity} \(\bar{\mbs{\Delta}}\) };
\end{tikzpicture}}%

    \vspace{-22pt}
    \caption{Loop transformation of the negative feedback interconnection representation of the GD method in \Cref{fig:gd_neg_feedback}. The loop transformation introduces a feedthrough term \(D\eye\), resulting in the modified GD controller \(\bar{\bm{\mathcal{G}}}_{\textrm{GD}}\), the nonlinearity \(\bar{\mbs{\Delta}}\), and \(\bar{\mbf{r}}_2^k = \mbf{r}_2^k - D\mbf{r}_1^k\).}
    \label{fig:loop_transformation}
    \vspace{-12pt}
\end{figure}
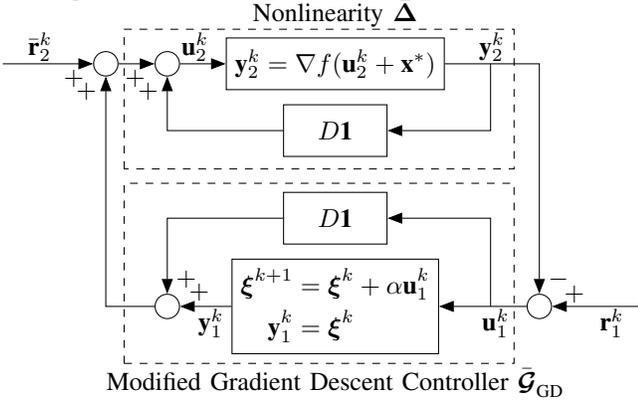
    The GD controller \(\bm{\mathcal{G}}_{\textrm{GD}}\), shown in \Cref{fig:gd_neg_feedback}, is strictly proper, since \(\mbf{D}=\mbf{0}\). However, a discrete-time system of the form \cref{eq:G} with \(\mbf{D} = \mbf{0}\) can never be passive~\mbox{\cite[Remark 2.7]{Byrnes_Lin_Losslessness}}. Therefore, to analyze the passivity of \(\bm{\mathcal{G}}_{\textrm{GD}}\), the loop transformation in \Cref{fig:loop_transformation} is performed, effectively introducing a constant feedthrough term \(\mbf{D} = D\eye\), where \(D \in \mathbb{R}_{>0}\). This feedthrough term results in the modified GD controller \(\bar{\bm{\mathcal{G}}}_{\textrm{GD}}\) with a minimal realization \(\mleft( \mbf{A}, \mbf{B}, \mbf{C}, \mbf{D} \mright) = \mleft( \eye, \alpha \eye, \eye, D\eye \mright)\). This loop transformation does not change the closed-loop dynamics of the system, meaning \Cref{fig:loop_transformation} still represents the GD method with update rule \cref{eq:GD}. Within the context of convex constrained optimization problems, similar loop transformation techniques have been used in~\cite{Notarnicola_passivity_admm,simpsons_loop_trans_pd_alg} to analyze the stability of discrete-time ADMM and primal-dual algorithms, respectively. Herein, the sufficient amount of the feedthrough term \(D\) required to render \(\bar{\bm{\mathcal{G}}}_{\textrm{GD}}\) passive is shown to be directly related to the step size \(\alpha\).
    \begin{lemma}\label{lemma:passivity_of_gd}
        The modified GD controller \(\bar{\bm{\mathcal{G}}}_{\textrm{GD}}\) in \Cref{fig:loop_transformation} is passive if and only if \(D \geq \alpha / 2 > 0\).
    \end{lemma}
    \begin{proof}
        Substituting the minimal realization of \(\bar{\bm{\mathcal{G}}}_{\textrm{GD}}\) into \cref{eqn:positive_real} yields
        \begin{equation}
            \mbf{M} =
            \begin{bmatrix}
                \mbf{0}                & \alpha \mbf{P} - \eye \\
                \alpha \mbf{P} - \eye  & \alpha^2 \mbf{P} - 2D\eye
            \end{bmatrix}.
        \end{equation}
        From the Schur complement lemma, \(\mbf{M} \preceq 0\) if and only if
        \begin{align}\label{eqn:pr_cond_1}
            \alpha \mbf{P} - \eye &= \mbf{0}, & \alpha^2 \mbf{P} - 2D\eye &\preceq 0.
        \end{align} 
        From \cref{eqn:pr_cond_1}, it follows that \(\mbf{P} = \frac{1}{\alpha}\) and~\mbox{\(D \geq \alpha / 2\)}, for \(\alpha \in \mathbb{R}_{>0}\). Finally, \Cref{lemma:positive_real} in tandem with \Cref{remark:positive_real_is_passive} concludes that \(\bar{\bm{\mathcal{G}}}_{\textrm{GD}}\) is passive if and only if \(D \geq \alpha / 2\), for \(\alpha \in \mathbb{R}_{>0}\). \hspace*{\fill}~\QED\par\endtrivlist\unskip
    \end{proof}

    Next, it is shown that for \(f \in \mathcal{S}_{m, L}\), the sector-bound condition in \cref{eqn:co_coercivity} leads to \(\mbs{\Delta}\) being VSP\@.
    \begin{lemma}\label{lemma:delta_vsp}
        Consider a function \(f \in \mathcal{S}_{m, L}\) with \(m, L \in \mathbb{R}_{>0}\), such that \(m \leq L\),
         and its unique global minimizer \(\mbf{x}^\ast\). The operator \(\mbs{\Delta}\) in \Cref{fig:gd_neg_feedback} mapping \(\mbf{u}_2^k\) to \(\mbf{y}_2^k\) as \(\mbf{y}_{2}^{k} = \nabla f(\mbf{u}_{2}^k + \mbf{x}^\ast)\) is VSP with
        \begin{align}\label{eqn:epsilon_delta}
            \varepsilon &= \frac{1}{m + L} > 0, & \delta &= \frac{mL}{m + L} > 0.%
        \end{align}
    \end{lemma}
    \vspace{5pt}
    \begin{proof}
        For a given discrete time interval \(\mathcal{T} = \mleft\{0, 1, \hdots, T-1 \mright\}\), with \(T \in \mathbb{Z}_{>0}\), consider the iterates \(\mleft\{ \mbf{y}_{2}^k \mright\}\) obtained from the operator \(\mbs{\Delta}\), such that \(\mbf{y}_{2}^k = \nabla f(\mbf{u}_{2}^k + \mbf{x}^\ast)\). For \(f \in \mathcal{S}_{m, L}\), the sector-bound inequality in \cref{eqn:co_coercivity} can be written as
        \begin{equation} \label{eqn:sector_bound_k}
            \begin{split}
                \langle \mbf{u}_{2}^k, \mbf{y}_{2}^k\rangle 
                \geq 
                \frac{mL}{m + L}  \lVert \mbf{u}_{2}^k \rVert_2^2 
                +
                \frac{1}{m + L}  \lVert \mbf{y}_{2}^k \rVert_2^2,
            \end{split}
        \end{equation}
        for all \(\mbf{u}_{2}^k \in \mathbb{R}^n\) and \(k \in \mathcal{T}\). Summing \cref{eqn:sector_bound_k} over \(k \in \mathcal{T}\) and defining \(\varepsilon\) and \(\delta\) as per \cref{eqn:epsilon_delta}, it follows that
        \begin{equation}
            \begin{split}
                \sum_{k \in \mathcal{T}} \langle \mbf{u}_{2}^k, \mbf{y}_{2}^k\rangle =
                \left\langle \mbf{y}_{2}, \mbf{u}_{2}\right\rangle_{T}
                &\geq 
                \delta \mleft\| \mbf{u}_{2} \mright\|_{2T}^2 
                +
                \varepsilon \mleft\| \mbf{y}_{2} \mright\|_{2T}^2,
            \end{split}
        \end{equation}
        for all \(\mbf{u} \in \ell_{2e}\) and \(T \in \mathbb{Z}_{>0}\). Therefore, \(\mbs{\Delta}\) is VSP with \(\varepsilon, \delta \in \mathbb{R}_{>0}\) and \(\beta = 0\). \hspace*{\fill}~\QED\par\endtrivlist\unskip  
    \end{proof}

    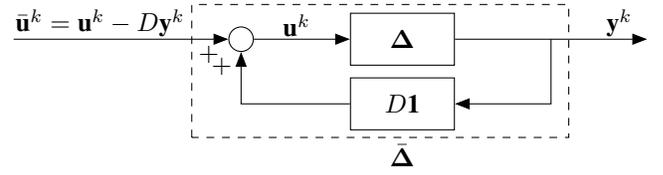
\begin{figure}
        \centering
        \vspace{3pt}
    \resizebox{1.0\columnwidth}{!}{\begin{tikzpicture}[auto, node distance=2cm]

    \def\blockWidth{4em}                 
    \def\blockHeight{2em}                
    \def\blockVerticalSpacing{3.5cm}     

    \def\arrowTipLength{2mm}             
    \def\arrowTipWidth{1.5mm}            

    \def\NdashedRectangleWidth{5.1cm}
    \def\NdashedRectangleHeight{1.8cm}
    \def\NdashedRectanglCenterXShift{3mm}
    \def\NdashedRectanglCenterYShift{-0.43cm}
    \def\Ndashedtextxhift{4mm}                 

    \def\smallArrowLength{1.3cm}           
    \def\smallArrowStretch{1.5}          
    \def\smallArrowShrink{0.8}           

    \tikzset{     
        block/.style = {draw, fill=none, rectangle, 
                        minimum height=\blockHeight, minimum width=\blockWidth}
    }

    \tikzset{    
        sum/.style  = {draw, fill=none, circle, node distance=1cm},
    }

    \tikzset{    
        tmp/.style    = {coordinate}, 
        input/.style  = {coordinate},
        output/.style = {coordinate}
    }

    \tikzset{      
        pinstyle/.style = {pin edge={to-,thin}}
    }

    \tikzset{     
        myarrow/.style = {->, -{Triangle[length=\arrowTipLength, width=\arrowTipWidth]}}
    }

    \node [block] (Nonlinearity) {\(\mbs{\Delta}\)};
    \node [block, below=\blockVerticalSpacing/4  of Nonlinearity.center, anchor=center] (D_sub) {\(D\eye\)};

    \node [tmp, right=\smallArrowLength  of Nonlinearity.east] (D_sub_in) {};
    \node [tmp, left=3.5*\smallArrowLength of Nonlinearity.west] (G_out) {};
    \node [tmp, right=2*\smallArrowLength of Nonlinearity.east] (G_in) {};

    \node [sum, left=\smallArrowLength of Nonlinearity.west] (sum_node_D_sub){};

    \draw[myarrow] (D_sub_in) |- (D_sub) {};

    \draw[myarrow] (sum_node_D_sub) -- (Nonlinearity)node[midway, above, xshift=-2pt, yshift=-0.09cm] {$\mbf{u}^{k}$}; {};

    \draw [myarrow] (D_sub) -| (sum_node_D_sub) node[at end, left, yshift=-\arrowTipLength] {$+$};

    \draw [myarrow] (G_out) -- node[midway, xshift=-8pt, yshift=-0.09cm]{$\bar{\mbf{u}}^{k} = \mbf{u}^{k} - D \mbf{y}^{k}$} (sum_node_D_sub) node[at end, left, yshift=-\arrowTipLength] {$+$};

    \draw [myarrow] (Nonlinearity) -- (G_in) node[at end, xshift=-\arrowTipLength-5pt, yshift=-0.09cm] {$\mbf{y}^{k}$};

    \draw[dashed] (Nonlinearity.center) ++(-\NdashedRectangleWidth/2 - \NdashedRectanglCenterXShift, -\NdashedRectangleHeight/2 + \NdashedRectanglCenterYShift) rectangle ++(\NdashedRectangleWidth, \NdashedRectangleHeight);
    \node[below, yshift= -\NdashedRectangleHeight/2 - \Ndashedtextxhift] at (Nonlinearity.center) {\(\bar{\mbs{\Delta}}\)};
\end{tikzpicture}}%

        \vspace{-18pt}
        \caption{A discrete-time system \(\bar{\mbs{\Delta}}\), composed of the positive feedback interconnection of a VSP system \(\mbs{\Delta}\), and the feedthrough term \(D\eye\).}
        \label{fig:feedback_interconnection}
        \vspace{-15pt}
    \end{figure}
    What remains is to derive an upper bound on \(D\) to ensure the positive feedback interconnection of \(\mbs{\Delta}\) and \(D\eye\) as per \Cref{fig:feedback_interconnection} results in a VSP system, \(\bar{\mbs{\Delta}}\).
    \begin{lemma} \label{lemma:upper_bound_D_vsp}
        Consider a VSP system \(\mbs{\Delta}\) with \(\varepsilon\) and \(\delta \) in \cref{eqn:epsilon_delta} and \(\beta = 0\). Given a feedthrough term \(D\eye\) with \(D \in \mathbb{R}_{>0}\), if \(D < 1/L\), the positive feedback interconnection of \(\mbs{\Delta}\) and \(D\eye\) results in a VSP system \(\bar{\mbs{\Delta}}\). 
    \end{lemma}
    \begin{proof}
        From the VSP property of \(\mbs{\Delta}\), it follows that
        \begin{align} \label{eqn:vsp_delta}
            \left\langle \mbf{u}, \mbf{y} \right\rangle_{T} 
            &\geq 
            \delta \mleft\| \mbf{u} \mright\|_{2T}^2 
            +
            \varepsilon \mleft\| \mbf{y} \mright\|_{2T}^2, \quad \forall \mbf{u} \in \ell_{2e}, \, \forall T \in \mathbb{Z}_{>0}.
        \end{align}
        Adding and subtracting the terms \(\delta D^2 \mleft\| \mbf{y} \mright\|_{2T}^2\) and \(2\delta D \left\langle \mbf{u}, \mbf{y} \right\rangle_{T}\) from \cref{eqn:vsp_delta} and completing the square, it follows that
        \begin{align}\label{eqn:vsp_delta_added_and_subtracted}
            \hspace{-5pt}
            \left\langle \mbf{u}, \mbf{y} \right\rangle_{T} 
            &\geq 
            \delta \mleft\| \bar{\mbf{u}} \mright\|_{2T}^2 
            +
            \mleft( \varepsilon - \delta D^2 \mright) \mleft\| \mbf{y} \mright\|_{2T}^2
            +
            2\delta D \left\langle \mbf{u}, \mbf{y} \right\rangle_{T},
        \end{align}
        where \(\bar{\mbf{u}}^{k} = \mbf{u}^{k} - D \mbf{y}^{k}\). Rearranging \cref{eqn:vsp_delta_added_and_subtracted} and collecting similar terms yields
        \begin{equation}
            \mleft( 1 - 2\delta D \mright) \left\langle \mbf{u}, \mbf{y} \right\rangle_{T} 
            \geq 
            \delta \mleft\| \bar{\mbf{u}} \mright\|_{2T}^2 
            +
            \mleft( \varepsilon - \delta D^2 \mright) \mleft\| \mbf{y} \mright\|_{2T}^2.
        \end{equation}
        For \(1 - 2\delta D  > 0\), it follows that 
        \begin{equation}\label{eqn:devided_by_1_minus_2deltaD}
            \left\langle \mbf{u}, \mbf{y} \right\rangle_{T} 
            \geq 
            \frac{\delta}{1 - 2\delta D} \mleft\| \bar{\mbf{u}} \mright\|_{2T}^2 
            +
            \frac{\mleft( \varepsilon - \delta D^2 \mright)}{1 - 2\delta D}
            \mleft\| \mbf{y} \mright\|_{2T}^2.
        \end{equation}
        \\[-10pt]
        Subtracting \(D \mleft\| \mbf{y} \mright\|_{2T}^2\) from both sides of \cref{eqn:devided_by_1_minus_2deltaD} and combining similar terms results in
        \begin{align}
            \left\langle \bar{\mbf{u}}, \mbf{y} \right\rangle_{T} 
            &\geq 
                \frac{\delta}{1 - 2\delta D} \mleft\| \bar{\mbf{u}} \mright\|_{2T}^2 
                +
                \frac{\mleft( \varepsilon - D +\delta D^2 \mright)}{1 - 2\delta D}
                \mleft\| \mbf{y} \mright\|_{2T}^2
            \\%
            &=
            \bar{\delta} \mleft\| \bar{\mbf{u}} \mright\|_{2T}^2
            +
            \bar{\varepsilon} \mleft\| \mbf{y} \mright\|_{2T}^2,
        \end{align}
        where \(\bar{\delta} = \delta / (1 - 2\delta D)\) and \(\bar{\varepsilon} = (\varepsilon - D + \delta D^2) / (1 - 2\delta D)\). For \(\bar{\mbs{\Delta}}\) to be VSP with \(\bar{\varepsilon}, \bar{\delta} \in \mathbb{R}_{>0}\), it is required that
        \begin{align}\label{eqn:upper_bound_D}
            D &< \frac{1}{2\delta}, & 0 &< \varepsilon - D + \delta D^2.
        \end{align}
        Note, \(D^\ast = 1/(2\delta)\) is the unique global minimum of the quadratic equation \(0 = \varepsilon - D + \delta D^2\).
        Substituting \mbox{\(\varepsilon = 1/(m + L)\)} and \(\delta = mL/(m + L)\) into \cref{eqn:upper_bound_D}, it follows that
        \begin{align}\label{eqn:upper_bound_D_1_2}
            D &< \frac{m + L}{2mL}, &
            0 &< \frac{1}{m + L} - D + \frac{mL}{m + L}D^2.
        \end{align}  
        The roots of the quadratic equation \cref{eqn:upper_bound_D_1_2} are \(r_1 = 1/L\) and \(r_2 = 1/m\). For \(m \leq L\), it follows that \(r_1 \leq D^\ast \leq r_2\), where equality is achieved for \(m = L\). For \(D < r_1 = 1/L\), both inequalities in \cref{eqn:upper_bound_D_1_2} are satisfied, and therefore, the positive feedback interconnection of \(\mbs{\Delta}\) and \(D\eye\) results in a VSP system \(\bar{\mbs{\Delta}}\), with \(\bar{\varepsilon}, \bar{\delta} \in \mathbb{R}_{>0}\). \hspace*{\fill}~\QED\par\endtrivlist\unskip 
    \end{proof}
    \begin{lemma}\label{lemma:upper_bound_D_ISP}
        Consider a VSP system \(\mbs{\Delta}\) with \(\varepsilon\) and \(\delta \) in \cref{eqn:epsilon_delta} and \(\beta = 0\). Given a feedthrough term \(D\eye\) with \(D \in \mathbb{R}_{>0}\), if \(D = 1/L\) and \(m < L\), the positive feedback interconnection of \(\mbs{\Delta}\) and \(D\eye\) results in an ISP system \(\bar{\mbs{\Delta}}\). 
    \end{lemma}
    \begin{proof}
        The proof follows similar to \Cref{lemma:upper_bound_D_vsp}, with the only difference being that for \(\bar{\mbs{\Delta}}\) to be ISP, with \(\bar{\delta} \in \mathbb{R}_{>0}\) and \(\bar{\varepsilon} = 0\), it is required that
            \begin{align} \label{eqn:upper_bound_D_ISP}
                D &< \frac{m + L}{2mL},
                &
                0 &= \frac{1}{m + L} - D + \frac{mL}{m + L}D^2.
            \end{align}  
        Therefore, to satisfy both conditions in \cref{eqn:upper_bound_D_ISP}, it is required that \(D = 1/L\) and \(m < L\). \hspace*{\fill}~\QED\par\endtrivlist\unskip
    \end{proof}

    Finally, the passivity theorem is invoked to analyze the input-output stability of the GD method represented as the negative feedback interconnection in \Cref{fig:loop_transformation}.
    \vspace{1pt}
    \begin{theorem}\label{thrm:io_stability_passivity_theorem}
        Consider the negative feedback interconnection representation of the GD method shown in \Cref{fig:loop_transformation}, where \(f \in \mathcal{S}_{m, L}\) with \( 0 < m \leq L\), and \(D = \alpha/2\). 
        \begin{enumerate}[label=\textit{\roman*})]
            \item {%
                Provided the step size \(\alpha \in (0, \frac{2}{L} )\) and \(\mbf{r}_1, \bar{\mbf{r}}_2 \in \ell_{2}\), then \(\mbf{y}_1 + D \mbf{u}_1 \in \ell_{2}\) and \(\mbf{y}_2 \in \ell_{2}\). \label{thrm:passivity_theorem_i}
            }%
            \item {%
                Additionally, for \(m < L\), provided the step size \(\alpha \in \mleft(0, \frac{2}{L}\mright]\), \(\bar{\mbf{r}}_2 = \mbf{0}\), and \(\mbf{r}_1 \in \ell_{2}\), then \(\mbf{y}_1 + D \mbf{u}_1 \in \ell_{2}\). \label{thrm:passivity_theorem_ii}
            }%
        \end{enumerate}
    \end{theorem}
    \vspace{2pt}
    \begin{proof}
        As per \Cref{lemma:passivity_of_gd}, \(\bar{\bm{\mathcal{G}}}_{\textrm{GD}}\) is passive if and only if \(0 < \alpha/2 \leq D\). Moreover, for \(f \in \mathcal{S}_{m, L}\), provided \(D < 1/L\), \(\bar{\mbs{\Delta}}\) is VSP, as per \Cref{lemma:upper_bound_D_vsp}. It follows that for \(D = \alpha/2\) with \(\alpha \in (0, 2/L)\), \(\bar{\bm{\mathcal{G}}}_{\textrm{GD}}\) is passive and \(\bar{\mbs{\Delta}}\) is VSP\@. The proof of \cref{thrm:passivity_theorem_i} then follows from \Cref{thrm:strong_passivity}. Furthermore, assuming \(m < L\) and \(D = \alpha/2 = 1/L\), \(\bar{\mbs{\Delta}}\) is ISP, as per \Cref{lemma:upper_bound_D_ISP}. Consequently, the proof of \cref{thrm:passivity_theorem_ii} then follows from \Cref{thrm:weak_passivity}, where for \(\alpha = 2/L\), \(\bar{\mbs{\Delta}}\) is ISP and \(\bar{\bm{\mathcal{G}}}_{\textrm{GD}}\) is passive. \hspace*{\fill}~\QED\par\endtrivlist\unskip 
    \end{proof}
    \section{Discussion and Extensions}\label{sec:discussion}
    The passivity-based analysis presented in \Cref{sec:main_results} decouples the stability analysis of the GD controller from the properties of the objective function. Notice, \Cref{lemma:passivity_of_gd} provides a lower bound on the feedthrough term \(D\) and dictates the necessary and sufficient condition for the controller to be passive. On the other hand, using the sector-bounds of the nonlinearity \(\mbs{\Delta}\), \Cref{lemma:delta_vsp,lemma:upper_bound_D_ISP,lemma:upper_bound_D_vsp} provide an upper bound on the feedthrough term \(D\). Therefore, for a function \(f \in \mathcal{S}_{m, L}\) with \(0 < m \leq L\), the passivity-based analysis of other first-order optimization algorithms can be approach as follows:
    \begin{enumerate}
        \item{%
            Cast the algorithm as an LTI controller in negative feedback with a static memoryless nonlinearity as per \Cref{fig:gd_neg_feedback}.
        }%
        \item{%
            Using the minimal realization of the controller, determine the lower bound on the feedthrough term \(D\) by solving the linear matrix
            inequality (LMI) in \cref{eqn:positive_real}.
        }%
        \item{%
            Determine the bounds on the algorithm's step size(s) that guarantee the feedthrough term \(D\) is within the bounds provided by \Cref{lemma:upper_bound_D_ISP,lemma:upper_bound_D_vsp}. Consequently, use the appropriate passivity theorem to guarantee the input-output stability of the algorithm.
        }%
    \end{enumerate}
    \subsection{Interpretation of Input-Output Stability}
        To draw a connection between the input-output stability of the negative feedback interconnection in \Cref{fig:loop_transformation} and the convergence of the GD method in \cref{eq:GD}, it is assumed that \(\mbf{r}_1^k = \mbf{r}_2^k = \mbf{0}\) for all \(k \in \mathbb{Z}_{\geq 0}\).
        
        For \(f \in \mathcal{S}_{m, L}\) with \( 0 < m \leq L\), \Cref{thrm:io_stability_passivity_theorem}-\cref{thrm:passivity_theorem_i} provides a sufficient condition to ensure the \(\ell_{2}\)-stability of \(\bar{\bm{\mathcal{G}}}_{\textrm{GD}}\) and \(\bar{\mbs{\Delta}}\). Firstly, if \(\mbf{y}_2 \in \ell_{2}\), then \(\mbf{y}_{2}^{k} \to \mbf{0}\) as \(k \to \infty\)~\cite{feedback_systems}. It follows that
        \begin{equation}\label{eqn:convergence_y2}
            \mbf{y}_{2}^{k} = \nabla f(\mbs{\xi}^k + \mbf{x}^\ast) = \nabla f(\mbf{x}^k ) \to \mbf{0}.
        \end{equation}
        Similarly, \(\mbf{y}_1 + D \mbf{u}_1 \in \ell_{2}\) implies that
        \begin{equation}\label{eqn:convergence_y1}
            \mbf{x}^k-D \nabla f(\mbf{x}^k ) \to \mbf{x}^\ast.
        \end{equation}
        From \cref{eqn:convergence_y2} and \cref{eqn:convergence_y1}, it follows that \(\mbf{x}^k \to \mbf{x}^\ast\). Therefore, the GD method is globally convergent for \(\alpha \in \mleft( 0, 2/L \mright)\). This result is consistent with Polyak's convergence analysis of the GD method in~\cite[Theorem 1]{Polyak}, where \(f\) is only assumed to be \(L\)-smooth and bounded from below. For \(f \in \mathcal{F}_{m, L} \), this convergence region is also obtained within the framework of equilibrium-independent dissipativity (EID) with quadratic supply rates in~\cite[Example 5.4]{simpson}.
        
        As discussed in~\cite{Polyak}, for \(f(x) = Lx^2/2\) with \(x \in \mathbb{R}\) and \(L \in \mathbb{R}_{>0}\), the choice of \(\alpha \geq 2/L\) results in \(f(\mbf{x}^{k+1}) \geq f(\mbf{x}^{k})\), for \(k \in \mathbb{Z}_{\geq 0}\). Here, for \(f \in \mathcal{S}_{m, L}\) with \(m < L\), \Cref{thrm:io_stability_passivity_theorem}-\cref{thrm:passivity_theorem_ii} can only guarantee \cref{eqn:convergence_y1}, which does not necessarily imply that \(\mbf{x}^k \to \mbf{x}^\ast\). In fact, inspired by~\cite{Polyak}, consider the objective function \(f(\mbf{x}) = \frac{1}{2}\mbf{x}^{\trans}\mbf{A}\mbf{x}\), where \(\mbf{x} = \begin{bmatrix} x_1 & x_2\end{bmatrix}^{\trans}\), \(\mbf{A} = \diag(m, L)\), and \(m < L\). It can be shown that for \(\alpha=2/L\), the GD method in \cref{eq:GD}, initialized using any \(\mbf{x}^0 \in \mathbb{R}^{2}\) leads to \(x_2^{k+1} = (-1)^{k+1}x_2^{k}\), for \(k \in \mathbb{Z}_{\geq 0}\). However, given that \(D = \alpha/2 = 1/L\), \cref{eqn:convergence_y1} still holds since \(x_2^k - (1/L)(Lx_2^k) = x_2^\ast = 0\). Therefore, although the series of iterations \(\{x_2^k\}\) produced by the GD method does not converge to \(x_2^\ast\), \cref{eqn:convergence_y1} still approaches \(x_2^\ast\). More generally, from \cref{eqn:convergence_y1}, it follows that
        \begin{equation}\label{eqn:new_stopping_criterion_1}
            \mleft( \mbf{x}^{k} - D \nabla f(\mbf{x}^{k}) \mright) - \mleft( \mbf{x}^{k-1} - D \nabla f(\mbf{x}^{k-1} ) \mright)  \to \mbf{0}.
        \end{equation}
        Since \(D=\alpha/2\), substituting the GD update rule \cref{eq:GD} into \cref{eqn:new_stopping_criterion_1} and simplifying yields \(\nabla f(\mbf{x}^{k}) + \nabla f(\mbf{x}^{k-1} ) \to \mbf{0}\).
        As such, in practice, the stopping criterion
        \begin{equation}
            \mleft\|\nabla f(\mbf{x}^{k}) + \nabla f(\mbf{x}^{k-1})\mright\|^{2}_{2} < \epsilon,
        \end{equation}
        for an arbitrarily small \(\epsilon \in \mathbb{R}_{>0}\), can be used to verify the relation in \cref{eqn:convergence_y1}.
    \subsection{Extension to Time Varying Step Sizes}\label{subsec:time_varying_step_sizes}
        \begin{figure}[t]
            \centering
            \vspace{1pt}
            \subfloat[Gain-scheduling a modified GD controller \(\bar{\bm{\mathcal{G}}}\) with step size \(\alpha = 1\). The scheduling function \(s^k\) is used to scale the input and output of \(\bar{\bm{\mathcal{G}}}\) such that \(\bar{\mbf{u}}^{k}_{1} = s^{k}\mbf{u}^{k}_1\) and \(\mbf{y}^{k}_{1} = s^{k}\bar{\mbf{y}}^{k}_{1}\).\vspace{-4pt}]{%
    \resizebox{1.0\columnwidth}{!}{\begin{tikzpicture}[auto, node distance=2cm]

    \def\blockWidth{4em}                 
    \def\blockHeight{2em}                
    \def\blockVerticalSpacing{1cm}       

    \def\arrowTipLength{2mm}             
    \def\arrowTipWidth{1.5mm}            

    \def\smallArrowLength{0.98cm}         
    \def\smallArrowStretch{1.5}          
    \def\smallArrowShrink{0.8}           

    \def\dashedRectangleHeight{1.8cm}
    \def\dashedRectangleWidth{7.25cm}
    \def\dashedRectanglCenterXShift{1.3mm}
    \def\dashedRectanglCenterYShift{7mm}

    \def\dashedRectangleHeightG{1.5cm}
    \def\dashedRectangleWidthG{3.6cm}
    \def\dashedRectanglCenterXShiftG{0mm}
    \def\dashedRectanglCenterYShiftG{0mm}

    \def\textPadding{0.25cm}              
    \def\textxhift{1.5mm}                 

    \tikzset{
        cross/.style = {path picture={ \draw[] (path picture bounding box.south east) -- 
                                                    (path picture bounding box.north west) 
                                                    (path picture bounding box.south west) -- 
                                                    (path picture bounding box.north east);}}
    }

    \tikzset{    
        sum/.style  = {draw, fill=none, circle, node distance=1cm},
        mult/.style = {draw, fill=none, circle, node distance=1cm,
                       append after command={\pgfextra{\let\TikZlastnode\tikzlastnode} 
                       node[cross, at=(\TikZlastnode)] {}}}
    }

    \tikzset{     
        block/.style = {draw, fill=none, rectangle, 
                        minimum height=\blockHeight, minimum width=\blockWidth},
        square/.style = {draw, fill=none, rectangle, 
                        minimum height=\blockHeight, minimum width=\blockHeight}
    }

    \tikzset{    
        tmp/.style    = {coordinate}, 
        input/.style  = {coordinate},
        output/.style = {coordinate}
    }

    \tikzset{      
        pinstyle/.style = {pin edge={to-,thin}}
    }

    \tikzset{     
        myarrow/.style = {->, -{Triangle[length=\arrowTipLength, width=\arrowTipWidth]}}
    }

    \node [output] (output) {};

    \node [block, draw=blue, right=3*\smallArrowLength of output.west](output) (G) {%
        \(
            \begin{aligned}
                \mbs{\xi}^{k+1} &= \mbs{\xi}^{k} + \bar{\mbf{u}}_{1}^{k}
                \\%
                \bar{\mbf{y}}_{1}^{k}     &= \mbs{\xi}^{k} + D\bar{\mbf{u}}_{1}^{k}
            \end{aligned}
        \)
    };

    \node [square, right=\smallArrowLength of G.east] (scheduling_r) {$s^k$};
    \node [square, left=\smallArrowLength  of G.west] (scheduling_l) {$s^k$};

    \node [input, right=\smallArrowLength of scheduling_r.east] (input) {};

    \draw [myarrow] (scheduling_r) -- node[above, midway, xshift=2pt, yshift=-0.09cm]{$\bar{\mbf{u}}_{1}^{k}$}(G);
    \draw [myarrow] (G) -- node[above, midway, xshift=2pt, yshift=-0.09cm]{$\bar{\mbf{y}}_{1}^{k}$}(scheduling_l);
    \draw [myarrow] (scheduling_l) -- node[anchor=south, xshift=-6pt, yshift=-0.09cm]{$\mbf{y}_{1}^k$}(output);
    \draw [myarrow] (input) -- node[above, midway, xshift=6pt, yshift=-0.09cm]{$\mbf{u}_{1}^k$}(scheduling_r);

    \node[below, yshift=-\blockHeight/2 - 5pt, xshift=-\dashedRectanglCenterXShiftG] at (G.center) {\textcolor{blue}{Modified GD Controller $\bar{\bm{\mathcal{G}}}$}};

    \node[above, yshift=\blockHeight/2 + 9pt, xshift=-\dashedRectanglCenterXShift] at (G.center) {Gain-Scheduled Modified GD Controller $\bar{\bm{\mathcal{G}}}_{\textrm{GS}}$};
    \node [tmp, below=(\blockVerticalSpacing + \blockHeight)/2 - \dashedRectanglCenterYShift of G.center, xshift=-\dashedRectanglCenterXShift] (sumSr)(center) {};
    \draw[dashed] (center) ++(-\dashedRectangleWidth/2, -\dashedRectangleHeight/2) rectangle ++(\dashedRectangleWidth, \dashedRectangleHeight);
\end{tikzpicture}}%
            \label{fig:scheduling_input_output}%
            }%
            \hspace{0pt}
            \subfloat[Gain-scheduling a GD controller \(\bm{\mathcal{G}}\) with step size \(\alpha = 1\). The scheduling function \(s^k\) is used to scale the input and output of \(\bm{\mathcal{G}}\). A feedthrough term \(\bar{D} \eye = \mleft( s^k \mright)^2 D \eye\) is added such that \(\mbf{y}_{1}^{k} = \bar{D} \mbf{u}_{1}^{k} + s^k\mbf{y}^k\).]{%
    \resizebox{1.0\columnwidth}{!}{\begin{tikzpicture}[auto, node distance=2cm]

    \def\blockWidth{4em}                 
    \def\blockHeight{2em}                
    \def\blockVerticalSpacing{0.15cm}       

    \def\arrowTipLength{2mm}             
    \def\arrowTipWidth{1.5mm}            

    \def\smallArrowLength{0.78cm}         
    \def\smallArrowStretch{1.5}          
    \def\smallArrowShrink{0.8}           

    \def\dashedRectangleHeight{2.7cm}
    \def\dashedRectangleWidth{7.45cm}
    \def\dashedRectanglCenterXShift{3.7mm}
    \def\dashedRectanglCenterYShift{7mm}

    \def\dashedRectangleHeightG{1.5cm}
    \def\dashedRectangleWidthG{3.6cm}
    \def\dashedRectanglCenterXShiftG{0mm}
    \def\dashedRectanglCenterYShiftG{0mm}

    \def\textPadding{0.25cm}              
    \def\textxhift{1.5mm}                 

    \tikzset{
        cross/.style = {path picture={ \draw[] (path picture bounding box.south east) -- 
                                                    (path picture bounding box.north west) 
                                                    (path picture bounding box.south west) -- 
                                                    (path picture bounding box.north east);}}
    }

    \tikzset{    
        sum/.style  = {draw, fill=none, circle, node distance=1cm},
        mult/.style = {draw, fill=none, circle, node distance=1cm,
                       append after command={\pgfextra{\let\TikZlastnode\tikzlastnode} 
                       node[cross, at=(\TikZlastnode)] {}}}
    }

    \tikzset{     
        block/.style = {draw, fill=none, rectangle, 
                        minimum height=\blockHeight, minimum width=\blockWidth},
        square/.style = {draw, fill=none, rectangle, 
                        minimum height=\blockHeight, minimum width=\blockHeight}
    }

    \tikzset{    
        tmp/.style    = {coordinate}, 
        input/.style  = {coordinate},
        output/.style = {coordinate}
    }

    \tikzset{      
        pinstyle/.style = {pin edge={to-,thin}}
    }

    \tikzset{     
        myarrow/.style = {->, -{Triangle[length=\arrowTipLength, width=\arrowTipWidth]}}
    }

    \node [output] (output) {};

    \node [block, draw=blue, right=4.5*\smallArrowLength of output.west](output) (G) {%
        \(
            \begin{aligned}
                \mbs{\xi}^{k+1} &= \mbs{\xi}^{k} + \bar{\mbf{u}}_{1}^{k}
                \\%
                \mbf{y}^{k}     &= \mbs{\xi}^{k}
            \end{aligned}
        \)
    };
    \node [block, above=\blockVerticalSpacing of G.north, anchor=south] (D) {\(\bar{D} \eye\)};

    \node [square, right=\smallArrowLength of G.east] (scheduling_r) {$s^k$};
    \node [square, left=\smallArrowLength  of G.west] (scheduling_l) {$s^k$};

    \node [input, right=1.5*\smallArrowLength of scheduling_r.east] (input) {};
    \node [sum, left=\smallArrowLength of scheduling_l.west] (sum) {};
    \node [tmp, right=\smallArrowLength/2 of scheduling_r.east] (temp) {};

    \draw [myarrow] (scheduling_r) -- node[above, midway, xshift=2pt, yshift=-0.09cm]{$\bar{\mbf{u}}_{1}^{k}$}(G);
    \draw [myarrow] (G) -- node[above, midway, xshift=3pt, yshift=-0.09cm]{$\mbf{y}^{k}$}(scheduling_l);
    \draw [myarrow] (scheduling_l) -- (sum)node[at end, right, yshift= \arrowTipLength] {$+$};;
    \draw [myarrow] (sum) -- node[anchor=south, yshift=-0.09cm]{$\mbf{y}_{1}^k$}(output);
    \draw [myarrow] (input) -- node[anchor=south, xshift=\smallArrowLength/2-1pt, yshift=-0.09cm]{$\mbf{u}_{1}^k$}(scheduling_r);
    \draw [myarrow] (D) -| (sum)node[at end, right, yshift= \arrowTipLength] {$+$};
    \draw [myarrow] (temp) |- (D){};

    \node[below, yshift=-\blockHeight/2 - 7pt, xshift=-\dashedRectanglCenterXShiftG] at (G.center) {\textcolor{blue}{GD Controller $\bm{\mathcal{G}}$}};

    \node[above, yshift=\blockHeight/2 + 1pt, xshift=-\dashedRectanglCenterXShift] at (D.center) {Gain-Scheduled Modified GD Controller $\bar{\bm{\mathcal{G}}}_{\textrm{GS}}$};
    \node [tmp, below=(\blockVerticalSpacing + \blockHeight)/2 - \dashedRectanglCenterYShift of G.center, xshift=-\dashedRectanglCenterXShift] (sumSr)(center) {};
    \draw[dashed] (center) ++(-\dashedRectangleWidth/2, -\dashedRectangleHeight/2) rectangle ++(\dashedRectangleWidth, \dashedRectangleHeight);
\end{tikzpicture}}%
                \label{fig:scheduling_input_output_add_feedthrough}%
            }%
            \caption{Two gain-scheduling architectures resulting in the same gain-scheduled modified GD controller \(\bar{\bm{\mathcal{G}}}_\textrm{GS}\) with minimal realization \(\mleft( \bar{\mbf{A}}, \bar{\mbf{B}}, \bar{\mbf{C}}, \bar{\mbf{D}} \mright) = (\eye, s^k\eye, s^k\eye, (s^k)^2 D\eye)\).}%
            \label{fig:gain_scheduling}
            \vspace{-10pt}
        \end{figure}
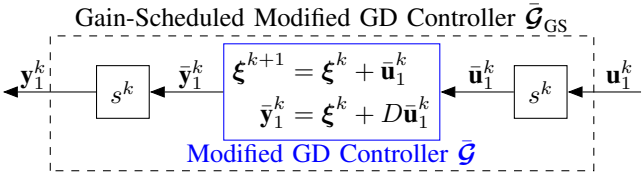
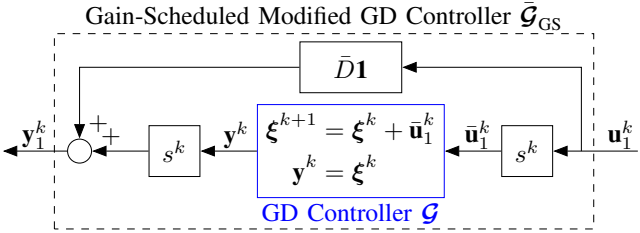
        This subsection introduces a new variation of the GD method with varying step size. The gain-scheduling architecture presented in~\cite{Damaren_passive_map} is adapted to discrete-time systems and used to schedule the input and output of the modified GD controller.
        
        Consider the modified GD controller \(\bar{\bm{\mathcal{G}}}\) in \Cref{fig:scheduling_input_output}, with a fixed step size \mbox{\(\alpha = 1\)}. As per \Cref{lemma:passivity_of_gd}, \(\bar{\bm{\mathcal{G}}}\) is passive if and only if \(D \geq 1/2\). From \Cref{def:passivity}, it follows that there exists a \(\beta \in \mathbb{R}_{\leq 0}\) such that \(\beta \leq \left\langle \bar{\mbf{y}}_1, \bar{\mbf{u}}_1 \right\rangle_T\) for all \(\bar{\mbf{u}}_1 \in \ell_{2e}\) and \(T \in \mathbb{Z}_{>0}\). This controller is gain-scheduled as per~\cite{Damaren_passive_map} using some general scheduling function \mbox{\(s^k \in \mathbb{R}\setminus {0}\)}, for \(k \in \mathcal{T} = \mleft\{0, 1, \hdots, T-1 \mright\}\), with \(T \in \mathbb{Z}_{>0}\). From the scheduling relation in \Cref{fig:scheduling_input_output}, it follows that
        \begin{equation*}
            \beta \leq 
            \left\langle \bar{\mbf{y}}_1, \bar{\mbf{u}}_1 \right\rangle_T 
            = 
            \left\langle \bar{\mbf{y}}_1, s\mbf{u}_1 \right\rangle_T 
            =
            \left\langle s\bar{\mbf{y}}_1, \mbf{u}_1 \right\rangle_T 
            =
            \left\langle \mbf{y}_1, \mbf{u}_1 \right\rangle_T.
        \end{equation*}
        Therefore, the gain-scheduling of a passive discrete-time system as per \Cref{fig:scheduling_input_output} results in a gain-scheduled passive system. As shown in \Cref{fig:scheduling_input_output_add_feedthrough}, \(\bar{\bm{\mathcal{G}}}_{\textrm{GS}}\) can also be constructed by adding a feedthrough term \(\bar{D} \eye = \mleft( s^k \mright)^2 D \eye\) to the gain-scheduled GD controller. It follows that \(\bar{\bm{\mathcal{G}}}_{\textrm{GS}}\) is passive if and only if \(\bar{D} \geq \mleft( s^k \mright)^2/2\). Moreover, the negative feedback interconnection of \(\bar{\bm{\mathcal{G}}}_{\textrm{GS}}\) and \(\bar{\mbs{\Delta}}\), as per \Cref{fig:loop_transformation} can then be used to represent a variation on the GD algorithm with the update rule
        \begin{equation}\label{eqn:gs_gradient_descent}
            \mbf{x}^{k+1} = \mbf{x}^{k} - s^{k}\nabla f(s^{k}\mbf{x}^{k}),
        \end{equation}
        where the scheduling function \(s^k \in \mathbb{R}\setminus {0}\) can be used to represent time varying step sizes. For a constant \(s^k = s\), using the change of variable \(\bar{\mbf{x}}^k = s\mbf{x}^k\), the update rule in~\cref{eqn:gs_gradient_descent} can be written as \(\bar{\mbf{x}}^{k+1}\ = \bar{\mbf{x}}^{k} - s^2\nabla f(\bar{\mbf{x}}^{k})\), recovering the original GD algorithm with step size \(\alpha = s^2\). For \(f \in \mathcal{S}_{m, L}\), \(\bar{\mbs{\Delta}}\) is VSP as per \Cref{lemma:upper_bound_D_vsp}, or ISP as per \Cref{lemma:upper_bound_D_ISP}, provided \(\bar{D} < 1/L\) or \(\bar{D}=1/L\), respectively. Consequently, using the lower bound \(\bar{D} \geq \mleft( s^k \mright)^2/2\), a similar analysis as \Cref{thrm:io_stability_passivity_theorem} shows that for \(m < L\), provided \(\max_{k \in \mathcal{T}} \lvert s^k \rvert \in (0, \sqrt{2/L}]\), then \(\mbf{y}_1^k \in \ell_{2}\).
    \section{Numerical Results} \label{sec:simulation}
    In this section, the performance of the GD method in \cref{eq:GD} is compared with the proposed gain-scheduled GD method in \cref{eqn:gs_gradient_descent}. The objective function of interest is taken from~\cite{ugrinovskii}, where \(f \in \mathcal{S}_{m, L}\) is defined as
    \begin{equation}\label{eqn:fx}
        f(x) = \frac{L - m}{4} \mleft( \frac{L + m}{L - m}x^2 + 2\sin \mleft( x \mright) - 2x\cos \mleft( x \mright) \mright),
    \end{equation}
    for \(m = 1\) and \(L = 100\). As shown in \Cref{fig:sector_bound}, the function \(f\) is non-convex but has a sector-bounded gradient and a unique global minimizer at \(x^\ast = 0\).
    \begin{figure}[t]
        \centering
        \vspace{3pt}
        \includegraphics{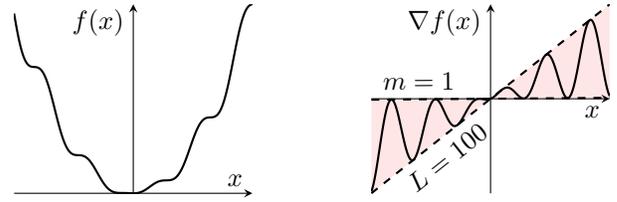}
        \caption{The function \(f \in \mathcal{S}_{m, L}\) defined in \cref{eqn:fx} for \(m=1\) and \(L=100\). This function is non-convex but has a sector-bounded gradient and a unique global minimizer at \(x^\ast = 0\).}
        \label{fig:sector_bound}
        \vspace{-10pt}
    \end{figure}
    
    To solve the unconstrained optimization problem in \cref{eq:optimization_problem} involving the function \(f\), the GD method in \cref{eq:GD} and the proposed gain-scheduled GD method in \cref{eqn:gs_gradient_descent} are implemented for various choices of step sizes and scheduling functions. Three choices of steps sizes are considered for the GD method. The first choice is a fixed step size of \(\alpha = 2/L\), being the largest possible step size to guarantee input-output stability for \(m < L\), as per \Cref{thrm:io_stability_passivity_theorem}-\cref{thrm:passivity_theorem_ii}. The second choice is the optimal fixed step size of \(\alpha = 2/(m + L)\)~\cite{lessard_dissipativity}. Finally, a varying step size, \(\alpha^{k}_{\textrm{btk}}\), is considered by implementing a line search algorithm based on Armijo's condition~\cite{Armijo}. Similarly, for the gain-scheduled GD method in \cref{eqn:gs_gradient_descent}, two choices of constant scheduling functions, \(s = \sqrt{2/L}\) and \(s = \sqrt{2/(m+L)}\), are considered. Additionally, a varying scheduling function, \(s^{k}_{\textrm{btk}}\), is considered by implementing a line search algorithm based on Armijo's condition with the added condition that \(\max_{k \in \mathcal{T}} \lvert s^k \rvert \in ( 0, \sqrt{2/L}]\). A Monte Carlo simulation is conducted with \(10^5\) uniformly sampled initial conditions \(x^0 \in [-10^5, 10^5]\) to compare the number of iterations needed to achieve \(\lvert \nabla f(x^k) \rvert < 10^{-12}\) for the GD method and the gain-scheduled GD method. 

    As shown in \Cref{fig:histogram} and \Cref{tab:a}, for the constant step size and constant scheduling function cases, the distribution of the gain-scheduled GD method has smaller mean, median, and mode values, compared to the standard GD method. For the varying step size and varying scheduling function case, the standard GD method has a lower mean, but a higher mode, compared to the gain-scheduled GD method. The lower mode indicates better performance in terms of fewer iterations. Finally, for both methods, \Cref{tab:a} also confirms that time-varying step sizes and scheduling functions can lead to improved performance.
    \begin{figure}[t]
        \centering
        \vspace{4pt}
        \includegraphics[width=\columnwidth]{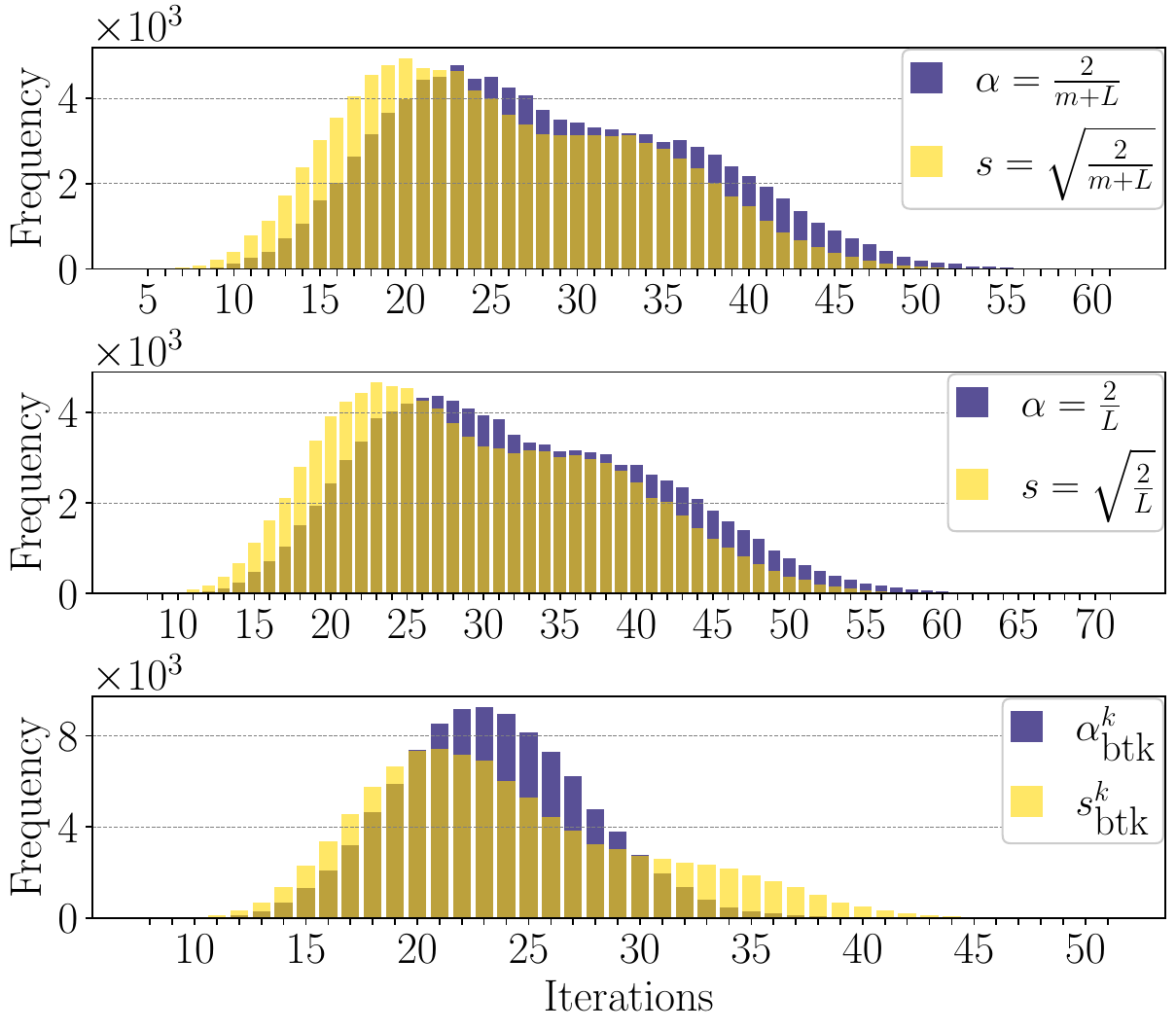}
        \vspace{-15pt}
        \caption{Comparison of number of iterations needed to achieve \(\lvert \nabla f(x^k) \rvert < 10^{-12}\) for the GD method and the gain-scheduled GD method with different step sizes and scheduling functions using \(10^5\) uniformly sampled initial conditions for \(x^0\in[-10^5, 10^5]\).}
        \label{fig:histogram}
        \vspace{5pt}
    \end{figure}
    \begin{table}[t]
        \caption{Corresponding Mean, Median, and Mode of \Cref{fig:histogram}}
        \vspace{-10pt}
        \label{tab:a}
        \fontsize{9pt}{11pt}\selectfont
        \begin{tabularx}{\columnwidth}{l *{3}{>{\centering\arraybackslash}X}}
            \hline
            \hline
            \addlinespace[1pt]
            {\fontsize{10pt}{12pt}\selectfont Parameter}
            & 
            {\fontsize{10pt}{12pt}\selectfont Mean}
            &
            {\fontsize{10pt}{12pt}\selectfont Median} 
            &
            {\fontsize{10pt}{12pt}\selectfont Mode} \\
            \hline
            \addlinespace[1pt]
            {\fontsize{8pt}{10pt}\selectfont$\alpha = 2/(m+L)$}        & {\(28.36\)} & {\(27\)} & {\(23\)} \\
            {\fontsize{8pt}{10pt}\selectfont$s = \sqrt{2/(m+L)}$}      & {\(\textbf{25.68}\)} & {\(\textbf{25}\)} & {\(\textbf{20}\)} \\
            \hline
            \addlinespace[1pt]
            {\fontsize{8pt}{10pt}\selectfont$\alpha = 2/L$}            & {\(32.12\)} & {\(31\)} & {\(27\)} \\
            {\fontsize{8pt}{10pt}\selectfont$s = \sqrt{2/L}$}          & {\(\textbf{29.50}\)} & {\(\textbf{28}\)} & {\(\textbf{23}\)} \\
            \hline
            {\fontsize{8pt}{10pt}\selectfont$\alpha^k_{\textrm{btk}}$} & {\(\textbf{23.38}\)} & {\(23\)} & {\(23\)} \\
            {\fontsize{8pt}{10pt}\selectfont$s^k_{\textrm{btk}}$}      & {\(24.08\)} & {\(23\)} & {\(\textbf{21}\)} \\
            \addlinespace[1pt]
            \hline
            \hline
        \end{tabularx}
        \vspace{-10pt}
    \end{table}
    \section{Closing Remarks}\label{sec:closing_remarks} 
    A discrete-time passivity-based analysis of the GD method is considered in this paper. For an objective function with sector-bounded gradient and a unique global minimizer, \(f \in \mathcal{S}_{m, L}\), it is shown that the shifted gradient \(\mbs{\Delta}\) is VSP\@. Using a loop transformation, it is shown that the GD method can be interpreted as a passive controller in negative feedback with a VSP system, provided \(\alpha \in (0, 2/L)\). The input-output stability, as well as the global convergence, of the GD method is then guaranteed using the strong version of the passivity theorem. Furthermore, provided \(m < L\), it is shown that for \(\alpha = 2/L\), the passive GD controller is now in negative feedback with an ISP system instead. Therefore, the weak passivity theorem is used to guarantee the input-output stability of the GD method. For this larger choice of step size, Polyak's counter example in~\cite{Polyak} is revisited, and a new stopping criterion is proposed to ensure the GD method can still be used to obtain the unique global minimizer. Finally, a new variation of the GD method is proposed, where the input and output of the gradient are scaled using a discrete-time adaptation of the gain-scheduling architecture in~\cite{Damaren_passive_map}.
    \printbibliography
\end{document}